\documentclass[11pt,a4paper,reqno]{amsart}
\usepackage[english]{babel}
\usepackage[T1]{fontenc}
\usepackage{verbatim}
\usepackage[shortlabels]{enumitem}
\usepackage{palatino}
\usepackage{amsmath}
\usepackage{mathabx}
\usepackage{amssymb}
\usepackage{amsthm}
\usepackage{amsfonts}
\usepackage{graphicx}
\usepackage{esint}
\usepackage{color}
\usepackage{mathtools}
\usepackage{overpic}

\usepackage[colorlinks = true, citecolor = black]{hyperref}
\author{Tuomas Orponen and Kevin Ren}
\title[Projections of almost Ahlfors regular sets]{On the projections of almost Ahlfors regular sets}
\address{Department of Mathematics and Statistics\\ University of Jyv\"askyl\"a,
P.O. Box 35 (MaD)\\
FI-40014 University of Jyv\"askyl\"a\\
Finland}
\email{tuomas.t.orponen@jyu.fi}
\address{Department of Mathematics, Princeton University, 304 Washington Rd, Princeton, NJ 08540}
\email{kevinren@princeton.edu}
\date{\today}
\thanks{T.O. is supported by the European Research Council (ERC) under the European Union’s Horizon Europe research and innovation programme (grant agreement No 101087499), and by the Research Council of Finland via the project \emph{Approximate incidence geometry}, grant no. 355453. K.R. is supported by an NSF GRFP fellowship.}

\newcommand{\R}{\mathbb{R}}

\newcommand{\N}{\mathbb{N}}
\newcommand{\Q}{\mathbb{Q}}

\newcommand{\Z}{\mathbb{Z}}

\newcommand{\spt}{\operatorname{spt}}

\newcommand{\diam}{\operatorname{diam}}

\newcommand{\m}{\mathfrak{m}}

\newcommand{\udelta}{\underline{\delta}}

\def\Barint_#1{\mathchoice
          {\mathop{\vrule width 6pt height 3 pt depth -2.5pt
                  \kern -8pt \intop}\nolimits_{#1}}%
          {\mathop{\vrule width 5pt height 3 pt depth -2.6pt
                  \kern -6pt \intop}\nolimits_{#1}}%
          {\mathop{\vrule width 5pt height 3 pt depth -2.6pt
                  \kern -6pt \intop}\nolimits_{#1}}%
          {\mathop{\vrule width 5pt height 3 pt depth -2.6pt
                  \kern -6pt \intop}\nolimits_{#1}}}

\numberwithin{equation}{section}

\theoremstyle{plain}
\newtheorem{thm}[equation]{Theorem}
\newtheorem{"thm"}[equation]{"Theorem"}

\newtheorem{lemma}[equation]{Lemma}
\newtheorem{"lemma"}[equation]{"Lemma"}

\newtheorem{cor}[equation]{Corollary}
\newtheorem{"proposition"}[equation]{"Proposition"}
\newtheorem{proposition}[equation]{Proposition}

\theoremstyle{definition}

\newtheorem{definition}[equation]{Definition}

\theoremstyle{remark}
\newtheorem{remark}[equation]{Remark}

\newcommand{\nref}[1]{(\hyperref[#1]{#1})}

\DeclareMathSymbol{\intop}  {\mathop}{mathx}{"B3}

\begin{document}

\begin{abstract} We show that the "sharp Kaufman projection theorem" from 2023 is sharp in the class of Ahlfors $(1,\delta^{-\epsilon})$-regular sets. This is in contrast with a recent result of the first author, which improves the projection theorem in the class of Ahlfors $(1,C)$-regular sets. \end{abstract}

\maketitle

%\tableofcontents

\section{Introduction}

This note demonstrates that projection theorems for Ahlfors $(1,\delta^{-\epsilon})$-regular sets cannot be stronger than projection theorems for general sets -- and in particular they are not as strong as available projection theorems for Ahlfors $(1,C)$-regular sets.
\begin{definition}[Ahlfors $(s,C)$-regularity]\label{def:sCregularity} Let $C,s > 0$. A Borel measure $\mu$ on $\R^{d}$ is \emph{Ahlfors $(s,C)$-regular} if
\begin{equation}\label{ADR} C^{-1}r^{s} \leq \mu(B(x,r)) \leq Cr^{s}, \qquad x \in \spt \mu, \, 0 < r \leq \diam (\spt \mu). \end{equation}
A closed set $K \subset \R^{d}$ is called Ahlfors $(s,C)$-regular if $K = \spt \mu$ for some non-trivial Ahlfors $(s,C)$-regular measure $\mu$. \end{definition}

In 2023, the first author and Shmerkin \cite{2023arXiv230110199O} obtained the following projection theorem for "almost" Ahlfors regular sets (only the $1$-dimensional case is stated for simplicity):

\begin{thm}[Corollary 4.9 in \cite{2023arXiv230110199O}]\label{cor1} For every $C > 0$ and $0 < t < \tau \leq 1$, there exist $\delta_{0},\epsilon > 0$ such that the following holds for all $\delta \in (0,\delta_{0}]$. Let $\mu$ be an Ahlfors $(1,\delta^{-\epsilon})$-regular probability measure with $K := \spt \mu \subset B(1)$, and let $\nu$ be a Borel probability measure on $[0,1]$ satisfying $\nu(B(x,r)) \leq Cr^{\tau}$ for all $x \in [0,1]$ and $r > 0$. Then, there exists $\theta \in \spt \nu$ such that
\begin{equation}\label{form17} N_\delta (\pi_{\theta}(K)) \geq \delta^{-(1 + t)/2}. \end{equation}  \end{thm}  
Here, $N_\delta (K)$ is the $\delta$-covering number of $K$, and $\pi_{\theta}(x,y) := x + \theta y$ for $(x,y) \in \R^{2}$, $\theta \in \R$. In \cite{2023arXiv230808819R}, the second author and Wang showed that the Ahlfors $(1,\delta^{-\epsilon})$-regularity of $\mu$ can relaxed to the Frostman condition $\mu(B(x,r)) \leq \delta^{-\epsilon}r$ (for $x \in \R^{2}$ and $r > 0$). This is the \emph{sharp Kaufman projection theorem} -- being the sharp version of Kaufman's projection theorem \cite{Ka} from the 60s.

As the name "sharp Kaufman" suggests, the projection theorem from \cite{2023arXiv230808819R} is sharp under the hypothesis $\mu(B(x,r)) \leq \delta^{-\epsilon}r$, or even $\mu(B(x,r)) \leq r$. This can be deduced from \cite[p. 10]{Wolff99} with some effort; more recent sources are \cite[Lemma 3.1]{2024arXiv240411179F} and \cite[Appendix]{MR4745881}. However, it remained open until now whether the numerology in \eqref{form17} could be improved in the class of $(1,\delta^{-\epsilon})$-regular measures. Our first main result shows that this not the case:
\begin{thm}\label{main1} For every $0 < \tau < t \leq 1$ with $t,\tau \in \Q$, there exists $\delta_{0} = \delta_{0}(t,\tau) > 0$ such that the following holds for all $\delta \in (0,\delta_{0}]$. There exists an Ahlfors $(1,\log(1/\delta)^{74/(t - \tau)})$-regular probability measure $\mu$ with $\spt \mu = K \subset B(1) \subset \R^{2}$ and a Borel probability measure $\nu$ on $[0,1]$ satisfying $\nu(B(x,r)) \leq Cr^{\tau}$ for all $x \in [0,1]$ and $r > 0$ (where $C > 0$ is an absolute constant) such that
\begin{equation}\label{form15} N_\delta (\pi_{\theta}(K)) \leq \delta^{-(1 + t)/2}, \qquad \theta \in \spt \nu. \end{equation}
\end{thm}

Theorem \ref{main1} shows that the numerology \eqref{form17} is sharp even if the hypothesis of Ahlfors $(1,\delta^{-\epsilon})$-regularity is upgraded to $(1,\log (1/\delta)^{C})$-regularity (for $C = C(t,\tau) \geq 1$). In contrast, a recent result \cite{2024arXiv241006872O} of the first author yields an improvement over \eqref{form17} in the class of Ahlfors $(1,C)$-regular measures (only the $1$-dimensional case is stated for simplicity):

\begin{thm}\label{mainOld} For every $C > 0$, $\tau \in (0,1]$, there exists $\delta_{0} > 0$ such that the following holds for all $\delta \in (0,\delta_{0}]$. Let $\mu$ be an Ahlfors $(1,C)$-regular probability measure on $\R^{2}$ with $K := \spt \mu \subset B(1)$. Let $\nu$ be a Borel probability measure on $[0,1]$ satisfying $\nu(B(x,r)) \leq Cr^{\tau}$ for all $x \in [0,1]$ and $r > 0$. Then, there exists $\theta \in \spt \nu$ such that 
\begin{displaymath} N_\delta (\pi_{\theta}(K)) \geq \delta^{-1 + \tau}. \end{displaymath} \end{thm}

\begin{remark} The proof of Theorem \ref{mainOld} is effective. It seems likely that it would yield a statement, where the constant "$C$" is replaced by an extremely slowly increasing function of $1/\delta$, perhaps a function which grows slower than any iterated logarithm. It would be interesting to determine the sharp growth of this function. Theorem \ref{main1} shows that $\log(1/\delta)^{C}$ grows too rapidly, but $\log \log (1/\delta)$ remains plausible as far as we know.  \end{remark}

We then state a second negative result which shows that even "strict" Ahlfors regularity of $\mu$ is not sufficient to yield any improvement over \eqref{form17} if the Frostman constant of the measure $\nu$ is allowed to grow like a logarithm of $1/\delta$. Again, positive results remain plausible if the Frostman constant of $\nu$ only grows like $\log \log (1/\delta)$.

\begin{thm}\label{main2} For every $0 < \tau < t \leq 1$ with $t,\tau \in \Q$, there exists $\delta_{0} = \delta_{0}(t,\tau) > 0$ such that the following holds for all $\delta \in (0,\delta_{0}]$. There exists an absolute constant $C > 0$, an Ahlfors $(1,C)$-regular probability measure $\mu$ with $\spt \mu = K \subset B(1) \subset \R^{2}$ and a Borel probability measure $\nu$ on $[0,1]$ satisfying $\nu(B(x,r)) \leq C\log(1/\delta)^{3\tau(1 - \tau)/(t - \tau)}r^{\tau}$ for all $x \in [0,1]$ and $r > 0$ such that
\begin{equation}\label{form16} N_\delta (\pi_{\theta}(K)) \leq \delta^{-(1 + t)/2}, \qquad \theta \in \spt \nu. \end{equation}
\end{thm} 

\begin{remark} The proofs of Theorems \ref{main1} and \ref{main2} show that even the unions $\bigcup_{\theta \in \spt \nu} \pi_{\theta}(K)$ satisfy the same $\delta$-covering number estimates \eqref{form15} and \eqref{form16}. We also mention that the sets $K$ in both Theorems \ref{main1}-\ref{main2} are product sets, and the set $K$ in Theorem \ref{main1} is self-similar. This does not contradict Hochman's projection theorem \cite{Ho} for self-similar sets, because the number of similitudes generating $K$ depends on $\delta$. Finally: the set $K$ in Theorem \ref{main2}, while a product set, and strictly Ahlfors regular, is not a product of two strictly Ahlfors regular sets. Indeed for products of strictly Ahlfors regular sets the positive result \cite[Theorem 3.6]{MR4388762} sometimes contradicts the numerology of Theorem \ref{main2}. \end{remark} 

\subsection*{Structure of the paper} Section \ref{s:preliminaries} contains preliminaries. Section \ref{s:main} contains the main construction, see Proposition \ref{mainProp}, which yields Theorem \ref{main1} immediately. The proof of Theorem \ref{main2} requires an additional rescaling argument, and is completed in Section \ref{s4}. 

\section{Preliminaries}\label{s:preliminaries}

Here we record some elementary results which will allow us to easily identify that certain measures are $(s,C)$-regular with a bound on the constant "$C$". Below, for $A \subset \R^{d}$ and $r \in 2^{-\N}$, the notation $\mathcal{D}_{r}(A)$ stands for dyadic cubes of side-length $r$ intersecting $A$, and $|A|_r := |\mathcal{D}_r (A)|$ is the dyadic $r$-covering number of $A$ (which is comparable to the standard $r$-covering number $N_r (A)$).

\begin{definition}[$(s,C)$-regularity between two scales] Let $0 < \delta < \Delta \leq 1$, $s \in [0,d]$, and $C \geq 1$. A set $P \subset [0,1]^{d}$ is called \emph{$(s,C)$-regular between scales $\delta$ and $\Delta$} if 
\begin{equation}\label{form6} C^{-1}(R/r)^{s} \leq |P \cap Q|_{r} \leq C(R/r)^{s}, \qquad \delta \leq r \leq R \leq \Delta, \, Q \in \mathcal{D}_{R}(P). \end{equation} 
Here $r,R \in 2^{-\N}$. \end{definition} 

The next proposition spells out the connection between $(s,C)$-regular sets (as above) and Ahlfors $(s,C)$-regular measures as in Definition \ref{def:sCregularity}.

\begin{proposition}\label{prop1} Let $s \in [0,d]$, $C \geq 1$, and $\delta \in (0,\tfrac{1}{2}]$. Assume that $P \subset [0,1]^{d}$ is $(s,C)$-regular between scales $\delta$ and $1$. Then there exists an Alhfors $(s,O_{d}(C^{2}))$-regular probability measure supported on the $\delta$-neighbourhood of $P$.
  \end{proposition} 
 
 \begin{proof}
 Let $\underline{\delta} \leq \delta$ be the largest dyadic scale such that $\sqrt{d} \cdot \underline{\delta} \leq \delta$, thus $2^k \cdot \underline{\delta} \ge \delta$ where $k := \lfloor \log (2\sqrt{d}) \rfloor \sim_d 1$. Let $\nu$ be a fixed (but arbitrary) $(s,A)$-regular probability measure on $[0,1]^{d}$, where $A \geq 1$ is an absolute constant (for example a self-similar measure generated by two similitudes with contraction ratios $2^{-1/s}$, or $\nu := \delta_{0}$ if $s = 0$). Define the measure $\mu$ in the statement of the proposition as follows. For each $Q \in \mathcal{D}_{\underline{\delta}}(P)$, let $\nu_{Q} := \frac{1}{|P|_{\underline{\delta}}} \cdot T_{Q}\nu$, where $T_{Q} \colon [0,1]^{d} \to Q$ is a rescaling map. Set
\begin{displaymath} \mu := \sum_{Q \in \mathcal{D}_{\underline{\delta}}(P)} \nu_{Q}. \end{displaymath}
Clearly, $\spt \mu \subset \mathcal{D}_{\udelta} (P)$ is contained within the $\sqrt{d} \cdot \underline{\delta} \le \delta$-neighbourhood of $P$.
We observe the following preliminary properties:
\begin{enumerate}[(P1)]
    \item\label{item:P1} $\mu([0,1]^{d}) = 1$;

    \item\label{item:P2} $\mu(Q) = \frac{|P \cap Q|_{\underline{\delta}}}{|P|_{\underline{\delta}}}$ for all $Q \in \mathcal{D}_{r}(P)$, $r \in [\underline{\delta}, 1] \cap 2^{-\N}$;

    \item\label{item:P3} $\mu(Q \cap B(x, r)) \lesssim_d A C r^s$ for all $Q \in \mathcal{D}_{\underline{\delta}}(P)$, $x \in \spt \mu$, and $r \le \sqrt{d} \cdot \udelta$, and additionally $\mu(Q \cap B(x, r)) \gtrsim_d (AC)^{-1} r^s$ if $x \in Q \cap \spt \mu$.
\end{enumerate}
Items \ref{item:P1} and \ref{item:P2} are clear. For \ref{item:P3}, we first record the important fact that
\begin{equation}\label{eqn:estimate P}
    C^{-1} \delta^{-s} \lesssim_d |P|_{\underline{\delta}} \lesssim_d C \delta^{-s}.
\end{equation}
Indeed, we first have $|P|_{2^k \udelta} \le |P|_{\udelta} \le 2^{kd} |P|_{2^k \udelta}$, and $|P|_{2^k \udelta}$ can in turn be estimated using \eqref{form6} with $r = 2^k \udelta$ and $R = 1$. Now, fix $Q \in \mathcal{D}_{\udelta} (P)$ and $x \in \spt \mu$: if $Q \cap B(x, r) = \emptyset$, then the first part of \ref{item:P3} holds trivially. If on the other hand $Q \cap B(x, r) \ni y$, then using the Ahlfors $(s,A)$-regularity of $\nu_Q$,
\begin{displaymath}
    \mu(Q \cap B(x, r)) \le \mu(Q \cap B(y, 2r)) \leq \frac{1}{|P|_{\udelta}} \cdot A \left(\frac{2r}{\udelta} \right)^s \stackrel{\eqref{eqn:estimate P}}{\lesssim_d} ACr^s.
\end{displaymath}
The second part of \ref{item:P3} is an easier application of the Ahlfors $(s,A)$-regularity of $\nu_{Q}$, and \eqref{eqn:estimate P}. Henceforth, we absorb the absolute constant $A$ into the implicit constants.

We first check the upper inequality for the $(s,O_{d}(C^{2}))$-regularity of $\mu$:
 \begin{equation}\label{form8} \mu(B(x,t)) \lesssim_{d} C^{2}t^{s}, \qquad x \in \spt \mu, \, t \le \sqrt{d}. \end{equation}

 Let first $x \in \R^{d}$ and $\underline{\delta} \leq t \leq \sqrt{d}$. Let $\underline{t} := \min\{ \max\{ t, \delta \}, 1 \} \sim_d t$, and pick $\ell \in (\underline{t}, 2\underline{t}] \cap 2^{-\N}$. Then $B(x,t)$ may be covered by $C_{d}$ dyadic cubes $Q$ of side-length $\ell(Q) = \ell$. Each of those cubes $Q$ has $\mu$ measure
 \begin{displaymath} \mu(Q) \stackrel{\textup{\ref{item:P2}}}{=} \frac{|P \cap Q|_{\underline{\delta}}}{|P|_{\underline{\delta}}} \le \frac{2^{kd} |P \cap Q|_{2^k\underline{\delta}}}{|P|_{\underline{\delta}}} \lesssim_{d} C^{2}\ell(Q)^{s} \lesssim_{d} C^{2} \underline{t}^{s} \lesssim_d C^2 t^s, \end{displaymath} 
 using \eqref{eqn:estimate P} and \eqref{form6} with $r = 2^k \underline{\delta}$ and $R = \ell(Q)$. Thus, $\mu(B(x,t)) \lesssim_{d} C^{2}t^{s}$.

 If $t \le \underline{\delta}$, then $B(x, t)$ can only intersect $\lesssim_d 1$ many cubes in $\mathcal{D}_{\underline{\delta}} (P)$. For each such $Q$, we use \ref{item:P3} to obtain $\mu(Q \cap B(x, t)) \lesssim_d C t^s$. Thus, $\mu(B(x, t)) \lesssim_d C t^s$ as well.
 
 Finally, we check lower inequality for the $(s,O_{d}(C^{2}))$-regularity of $\mu$:
 \begin{equation}\label{form8'}
     C^{-2}t^{s} \lesssim_{d} \mu(B(x,t)), \qquad x \in \spt \mu, \, t \le \sqrt{d}.
 \end{equation}
 Fix $x \in \spt \mu$ and first suppose $2\delta\sqrt{d} \leq t \leq 1$. Then $B(x,t)$ contains a dyadic cube $Q \in \mathcal{D}_{r}(P)$ with $r \in [\frac{t}{2\sqrt{d}}, \frac{t}{\sqrt{d}}] \cap 2^{-\N}$. Therefore, we find by \eqref{eqn:estimate P},
 \begin{displaymath} \mu(B(x,t)) \geq \mu(Q) = \frac{|P \cap Q|_{\underline{\delta}}}{|P|_{\underline{\delta}}} \ge \frac{|P \cap Q|_{2^k \underline{\delta}}}{|P|_{\underline{\delta}}} \gtrsim_{d} C^{-2}\ell(Q)^{s} \gtrsim_{d} C^{-2}t^{s}. \end{displaymath} 
 Finally, if $t \leq 2\delta \sqrt{d}$, then let $\underline{t} = \min\{ t, \delta \sqrt{d} \} \sim t$. We have $\mu(B(x, t)) \ge \mu(B(x, \underline{t})) \gtrsim_d C^{-1} t^s$ by applying \ref{item:P3} to the dyadic cube $Q \in \mathcal{D}_{\underline{\delta}} (P)$ containing $x$. In either case, we have proven \eqref{form8}.
% the case where $x \in \spt \mu_{0}$ and $\underline{\delta} \leq t \leq \delta \sqrt{d}$ can be reduced to the previous case by noting that $\mu_{0}(B(x,t)) \sim_{d} \mu_{0}(B(x,\sqrt{d}\max\{t,\delta\}))$. We have proven \eqref{form8}.
\end{proof}
 
In the following, for $\m \in \N_{\geq 2}$, $r \in \m^{\Z}$, and $A \subset \R^{d}$, the notation $\mathcal{D}_{r}^{\m}(A)$ stands for the $\m$-adic cubes of side-length $r$ in $\R^{d}$ intersecting $A$. We also write $|A|_{r}^{\m} := |\mathcal{D}_{r}^{\m}(A)|$. Note the following comparison between dyadic and $\m$-adic covering numbers: if $r \in 2^{\Z}$ and $\underline{r} \in \m^{\Z}$ with $\underline{r} \leq r$, then
 \begin{equation}\label{form18} |A|_{r} \lesssim_{d} |A|_{\underline{r}}, \qquad A \subset \R^{d}, \end{equation} 
since every element of $\mathcal{D}^{\m}_{\underline{r}}(A)$ can be covered by $\lesssim_{d} 1$ elements of $\mathcal{D}_{r}(A)$.
 
\begin{definition}[Uniform sets]\label{def:uniformity} Let $\m,n \geq 2$, and let $\{\Delta_{j}\}_{j = 0}^{n} \subset \m^{-\N}$ with
\begin{displaymath} \Delta_{n} < \Delta_{n - 1} < \ldots < \Delta_{1} \leq \Delta_{0} := 1. \end{displaymath}
A set $P \subset [0,1]^{d}$ is called \emph{$\{\Delta_j,\m\}_{j=0}^{n - 1}$-uniform} if there is a sequence $\{N_j\}_{j=0}^{n - 1} \subset \N^{n}$ such that $|P \cap Q|^{\m}_{\Delta_{j + 1}} = N_j$ for all $j\in \{0,\ldots,n - 1\}$ and all $Q\in\mathcal{D}_{\Delta_{j}}^{\m}(P)$. \end{definition}

\begin{lemma}\label{lemma2} Let $\m \in \N$, $\Delta \in \m^{-\N}$, and let $P \subset [0,1)^{d}$ be $\{\Delta^{j},\m\}_{j = 0}^{n - 1}$-uniform. Assume moreover that $N_{j} \equiv \Delta^{-s}$ for all $0 \leq j \leq n - 1$. Then $P$ is $(s,O_{d}(\Delta^{-3s}))$-regular between scales $\delta := \Delta^{n}$ and $1$. \end{lemma}

\begin{proof} Clearly, if $a \in \{0,\ldots,n - 1\}$ and $Q \in \mathcal{D}_{\Delta^{a}}(P)$, then
\begin{equation}\label{form7} |P \cap Q|_{\Delta^{b}}^{\m} = (\Delta^{a - b})^{s}, \qquad b \in \{a,\ldots,n\}.  \end{equation}
This looks like \eqref{form6}, except that \eqref{form6} is a statement about dyadic covering numbers and dyadic scales $(r,R)$, whereas \eqref{form7} only concerns the special (not necessarily dyadic) pairs $(r,R) = (\Delta^{b},\Delta^{a})$. For $(r,R)$ with $r,R \in 2^{-\N}$ and $\delta \leq r \leq R \leq 1$, choose $a,b \in \{1,\ldots,n\}$ such that $a \leq b$ and $\Delta^{b} \leq r \leq \Delta^{b - 1}$ and $\Delta^{a} \leq R \leq \Delta^{a - 1}$. For $Q \in \mathcal{D}_{R}(P)$, note that $P \cap Q$ can be covered by $C_{d}$ elements $\overline{Q} \in \mathcal{D}^{\m}_{\Delta^{a - 1}} (P)$. Moreover, for each such $\overline{Q} \in \mathcal{D}^{\m}_{\Delta^{a - 1}} (P)$,
\begin{displaymath} |P \cap \overline{Q}|_{r} \stackrel{\eqref{form18}}{\lesssim_{d}} |P \cap \overline{Q}|^{\m}_{\Delta^{b}} \stackrel{\eqref{form7}}{=} (\Delta^{a - 1 - b})^{s} \leq \Delta^{-2s}(R/r)^{s}. \end{displaymath} 
To prove the matching lower bound, first assume that $r,R \in [\Delta^{a},\Delta^{a - 2}]$ for some $a \in \{2,\ldots,n\}$. In this case $R/r \leq \Delta^{-2}$ and for $Q \in \mathcal{D}_{R}(P)$ we may use the trivial estimate $|P \cap Q|_{r} \geq 1 \geq \Delta^{2s}(R/r)^{s}$. Assume finally that $r \in [\Delta^{b},\Delta^{b - 1}]$ and $R \in [\Delta^{a},\Delta^{a - 1}]$ for some $a \leq b - 2$. Fix $Q \in \mathcal{D}_{R}(P)$. Since $R \geq \Delta^{a}$, certainly $R \geq 2\Delta^{a + 1}$. This implies that we can find a cube $\underline{Q} \in \mathcal{D}^{\m}_{\Delta^{a + 1}}$ with $\underline{Q} \subset Q$ (take the element of $\mathcal{D}^{\m}_{\Delta^{a + 1}}$ containing the centre of $Q$). Since $a + 1 \leq b - 1$,
\begin{displaymath} |P \cap Q|_{r} \stackrel{\eqref{form18}}{\gtrsim_{d}} |P \cap \underline{Q}|^{\m}_{\Delta^{b - 1}} \stackrel{\eqref{form7}}{=} (\Delta^{a - b + 2})^{s} \geq \Delta^{3s}(R/r)^{s},  \end{displaymath}
as claimed. \end{proof}

\begin{cor}\label{lemma1} Let $s \in [0,1]$, and let $\m \in \N_{\geq 2}$ be such that also $\m^{s} \in \N$. For each $0 \leq j \le \m^{s}-1$, pick a number $t_{j} := \m^{-1} k_{j}$ in such a way that $k_{j} \in \{0,\ldots,\m - 1\}$, and $k_{i} \neq k_{j}$ for $i \neq j$. Let $S \subset [0,1]$ be the self-similar set generated by the similitudes  $\psi_{j} \colon [0,1] \to [0,1]$, 
\begin{displaymath} \psi_{j}(x) = \m^{-1} x + t_{j}, \qquad 0 \leq j \leq \m^{s}-1. \end{displaymath}
Then $S$ is $(s,O(\m^{3s}))$-regular between the scales $\m^{-n}$ and $1$, for all $n \in \N$. \end{cor}

\begin{proof} The set $S$ is $\{\m^{-j},\m\}_{j = 0}^{n - 1}$-uniform with branching numbers $N_{j} \equiv \m^{s}$, so the claim follows immediately from Lemma \ref{lemma2}.  \end{proof}

\section{The main construction}\label{s:main}

The next proposition contains the construction used to prove both Theorems \ref{main1}-\ref{main2}. In fact, the proposition implies Theorem \ref{main1} immediately (take $K := A \times B$), but Theorem \ref{main2} will require an additional argument. In the statement, a Borel measure $\nu$ on $\R^{d}$ is called \emph{$(\tau,C)$-Frostman} if $\nu(B(x,r)) \leq Cr^{\tau}$ for all $x \in \R^{d}$ and $r > 0$.

\begin{proposition}\label{mainProp} For every $0 < \tau < t \leq 1$ with $\tau,t \in \Q$, there exist $\delta_{0} = \delta_{0}(t,\tau) > 0$ such that the following holds for all $\delta \in (0,\delta_{0}]$. There exist an absolute constant $C > 0$,
\begin{itemize}
\item a $(\tfrac{1 + \tau}{2},\log(1/\delta)^{37/(t - \tau)})$-regular probability measure $\mu_{A}$ with $A := \spt \mu_{A} \subset [0,1]$,
\item a $(\tfrac{1 - \tau}{2},\log(1/\delta)^{37/(t - \tau)})$-regular probability measure $\mu_{B}$ with $B := \spt \mu_{B} \subset [0,1]$,
\item a $(\tau,C)$-Frostman Borel probability measure $\nu$ with $\Theta := \spt \nu \subset [0,1]$
\end{itemize}
such that
\begin{equation}\label{form5} N_\delta (A + \Theta B) \leq \delta^{-(1 + t)/2}. \end{equation}
\end{proposition} 

It would be convenient to assume that some powers of $\log(1/\delta)$ are integers. In such a case we can actually prove a version of Proposition \ref{mainProp} with slightly better constants:

\begin{lemma}\label{mainLemma} For every $0 < \tau < t \leq 1$ with $\tau,t \in \Q$, define $\epsilon := \frac{t - \tau}{6}$. There exist $\delta_{0} = \delta_{0}(\epsilon) > 0$ such that if $\delta \in (0,\delta_{0}]$ satisfies
\begin{equation}\label{form3} \log(1/\delta), (\log(1/\delta))^{1/\epsilon}, (\log (1/\delta))^{\tau/\epsilon}, (\log (1/\delta))^{(1 + \tau)/(2\epsilon)}, (\log (1/\delta))^{(1 - \tau)/(2\epsilon)} \in \N_{\ge 2},
\end{equation}
then there exists an absolute constant $C > 0$,
\begin{itemize}
\item a $(\tfrac{1 + \tau}{2},C\log(1/\delta)^{6/\epsilon})$-regular probability measure $\mu_{A}$ with $A := \spt \mu_{A} \subset [0,1]$,
\item a $(\tfrac{1 - \tau}{2},C\log(1/\delta)^{6/\epsilon})$-regular probability measure $\mu_{B}$ with $B := \spt \mu_{B} \subset [0,1]$,
\item a $(\tau,C)$-Frostman Borel probability measure $\nu$ with $\Theta := \spt \nu \subset [0,1]$
\end{itemize}
such that
\begin{equation}\label{form5'} N_\delta (A + \Theta B) \leq \delta^{-(1 + t)/2 + \epsilon/2}. \end{equation}
\end{lemma}

To deduce Proposition \ref{mainProp} from Lemma \ref{mainLemma}, observe that \eqref{form3} has solutions of the form 
\begin{displaymath} \delta = 2^{-n^K}, \qquad n \in \N_{\ge 2}, \end{displaymath}
where $K = K(t, \tau)$ is the least common multiple of the denominators of the rational numbers $\frac{1}{2\epsilon}$ and $\frac{\tau}{2\epsilon}$. Now given $\delta \in (0,2^{-2^{K}}]$, we may find $n \geq 2$ such that $2^{n^K} \le \delta^{-1} \le 2^{(n+1)^K}$. The key point is that for $n \ge n_{0} := n_0 (t, \epsilon, K)$,
\begin{displaymath} (n+1)^K \le \left(1 + \frac{\epsilon}{1+t-\epsilon}\right) \cdot n^K. \end{displaymath}
Thus, for $\delta \leq 2^{-n_{0}^{K}}$, we have $\delta \ge \delta' := 2^{-(n+1)^K} \ge \delta^{1+\epsilon/(1+t-\epsilon)}$. Finally, Lemma \ref{mainLemma} applied to $\delta'$ implies Proposition \ref{mainProp} for $\delta$, provided $\delta_0$ is small enough.

\begin{proof}[Proof of Lemma \ref{mainLemma}]
Choose $\delta_0 (\epsilon) > 0$ so small that for any $\delta < \delta_0$,
\begin{equation}\label{eqn:smallness}
    \log \log (1/\delta) \ge 4 \quad \text{ and } \quad \log(1/\delta) \le \delta^{-\epsilon^2/2}.
\end{equation}
We also define
%Fix $0 < \tau < t \leq 1$, and let $\epsilon = \frac{t-\tau}{4} \in (0, \frac{1}{2}]$.
%There exist values $\delta \in 2^{-\N}$ satisfying the requirements above, since $t,\tau \in \Q$. The case of general (sufficiently small) $\delta > 0$ can be easily reduced to those satisfying \eqref{form3}, at the cost of making $\epsilon$ slightly worse.
\begin{displaymath} \ell(\delta) := \frac{\log \log (1/\delta)}{\epsilon \log(1/\delta)} \quad \text{and} \quad \rho := \delta^{\ell(\delta)} = \log(1/\delta)^{-1/\epsilon} \in (0,\tfrac{1}{2}]. \end{displaymath}
Consider the following three self-similar sets $A,B,\Theta \subset [0,1]$. Let $A \subset [0,1]$ be the self-similar set generated by the similitudes
\begin{displaymath} \varphi_{j}(x) := \rho x + j\rho^{(1 + \tau)/2}, \qquad 0 \leq j \leq \rho^{-(1 + \tau)/2} - 1. \end{displaymath} 
Let $B \subset [0,1]$ be the self-similar set generated by the similitudes
\begin{displaymath} \bar{\varphi}_{j}(x) := \rho x + j \rho^{(1 - \tau)/2}, \qquad 0 \leq j \leq \rho^{-(1 - \tau)/2} - 1. \end{displaymath}
Let $\Theta \subset [0,1]$ be the self-similar set generated by the similitudes
 \begin{equation}\label{form14} \psi_{j}(x) := \rho x + j\rho^{\tau}, \qquad 0 \leq j \leq \rho^{-\tau} - 1. \end{equation}
Note that $\mathfrak{m} = \rho^{-1} \in \N_{\geq 2}$ by hypothesis \eqref{form3}. Therefore, by Corollary \ref{lemma1} applied with $n := \lceil \ell(\delta)^{-1} \rceil$, $s := (1 + \tau)/2$, and $t_{j} = j\rho^{-(1 + \tau)/2} \in \N$, the set $A$ is $(\tfrac{1 + \tau}{2},C\log(1/\delta)^{3/\epsilon})$-regular between the scales $\rho^{n} \le \delta$ and $1$. Then, by Proposition \ref{prop1} the $\delta$-neighbourhood $[A]_{\delta}$ supports a $(\tfrac{1 + \tau}{2},C\log(1/\delta)^{6/\epsilon})$-regular probability measure. The proof below will show that \eqref{form5} holds for $[A]_{\delta}$, hence for $A' = \spt \mu$.
 
 By the same argument, $B$ is $(\tfrac{1 - \tau}{2},C\log(1/\delta)^{3/\epsilon})$-regular between scales $\delta$ and $1$, and $[B]_{\delta}$ supports a $(\tfrac{1 - \tau}{2},C\log(1/\delta)^{6/\epsilon})$-regular probability measure. 
 
 Regarding $\Theta$, we claim that the canonical equal-weights self-similar measure on $\nu$ on $\Theta$ satisfies $\nu(B(x,r)) \lesssim r^{\tau}$ for all $x \in [0,1]$ and $r > 0$. By self-similarity, it suffices to check this for $\rho \leq r \leq 1$. Note that $\nu$ is supported on the union of the intervals $I_{j} = [j\rho^{\tau},j\rho^{\tau + 1}]$ with $0 \leq j \leq \rho^{-\tau} - 1$, and $\nu(I_{j}) \equiv \rho^{\tau}$. Further, $B(x,r)$ intersects $\lesssim \max\{r/\rho^{\tau},1\}$ intervals $I_{j}$ for $\rho \leq r \leq 1$. Therefore $\nu(B(x,r)) \lesssim \max\{r,\rho^{\tau}\} \leq r^{\tau}$, as claimed.

We then begin to study the set $A + \Theta B$ from \eqref{form5}. To begin with, observe that $A,B$ and $\Theta$ can be expressed in the following way:
\begin{equation}\label{defA} A := \left\{ \sum_{j = 0}^{\infty} \rho^{j} \cdot k_{j}\rho^{(1 + \tau)/2} : k_{j} \in \N, \, 0 \leq k_{j} \leq \rho^{-(1 + \tau)/2} - 1 \right\}, \end{equation}
and
\begin{displaymath} B := \left\{ \sum_{j = 0}^{\infty} \rho^{j} \cdot l_{j}\rho^{(1 - \tau)/2} : l_{j} \in \N, \, 0 \leq l_{j} \leq \rho^{-(1 - \tau)/2} - 1 \right\}, \end{displaymath}
and
\begin{displaymath} \Theta := \left\{ \sum_{j = 0}^{\infty} \rho^{j} \cdot r_{j}\rho^{\tau} : r_{j} \in \N, \, 0 \leq r_{j} \leq \rho^{-\tau} - 1 \right\}. \end{displaymath} 
 We claim that $A + \Theta B$ is contained in the $O(\delta)$-neighbourhood of the following set $H$:
\begin{displaymath} H := \left\{\sum_{k = 0}^{n} \rho^{k} \cdot s_{k}\rho^{(1 + \tau)/2} : s_{k} \in \N, \, 0 \leq s_{k} \leq 2n \cdot \rho^{-(1 + \tau)/2} \right\},
\end{displaymath}
where still
\begin{displaymath} n = \lceil \ell(\delta)^{-1} \rceil = \lceil \epsilon \log(1/\delta)/\log \log(1/\delta) \rceil. \end{displaymath}
We note that since $\epsilon \le \frac{1}{2}$ and $\log(1/\delta) \ge 2 \log \log (1/\delta)$ (from \eqref{eqn:smallness}), we get
\begin{equation}\label{eqn:property of n}
    n \le \frac{\epsilon \log(1/\delta)}{\log \log(1/\delta)} + 1 \le \frac{\log(1/\delta)}{\log \log(1/\delta)}.
\end{equation}
Let us use this to complete the proof of \eqref{form5}. From \eqref{eqn:smallness}, \eqref{eqn:property of n}, and definition of $\rho$, we see that $\rho^{n} \ge \rho \cdot \delta \ge \delta^{1+\epsilon/2}$. Thus, we get the following estimate for the cardinality of $H$:
\begin{displaymath} |H| \leq (2n \cdot \rho^{-(1+ \tau)/2})^{n} \le (2n)^{n} \cdot \delta^{-(1 + \tau)/2 - \epsilon/2}. \end{displaymath}
Now, we have by \eqref{eqn:property of n} (since $\log \log (1/\delta) \ge 2$ from \eqref{eqn:smallness}),
\begin{displaymath} (2n)^{n} \le \left(\frac{2\log (1/\delta)}{\log \log (1/\delta)} \right)^{\epsilon \log (1/\delta)/\log \log (1/\delta) + 1} \leq \log(1/\delta)^{\epsilon \log(1/\delta)/\log \log (1/\delta) + 1} \le \delta^{-3\epsilon/2}. \end{displaymath}
Therefore, $(2n)^{n} \leq \delta^{-3\epsilon/2} \leq \delta^{3\epsilon/2 + (\tau - t)/2}$. This shows that $|H| \leq \delta^{-(1 + t)/2 + \epsilon}$. Since $N_\delta (A + \Theta B) \lesssim |H|$, finally $N_\delta (A + \Theta B) \leq \delta^{-(1/2 + t) + \epsilon/2}$ for $\delta > 0$ small enough.

To prove the inclusion $A + \Theta B \subset [H]_{C\delta}$, fix 
\begin{displaymath} \theta = \sum_{i = 0}^{\infty} \rho^{i} \cdot r_{i}\rho^{\tau} \quad \text{and} \quad b = \sum_{j = 0}^{\infty} \rho^{j} \cdot l_{j}\rho^{(1 - \tau)/2}, \end{displaymath}
Note that here $r_{i}\rho^{\tau} \leq 1$ and $l_{j}\rho^{(1 - \tau)/2} \leq 1$. Furthermore $\rho^{n} \le \delta$, and $\rho \leq \tfrac{1}{2}$, so
\begin{equation}\label{form9} \sum_{k \geq n + 1} \rho^{k} = \frac{\rho^{n + 1}}{1 - \rho} \leq 2\delta \cdot \rho \le \delta. \end{equation}
This shows that $\theta$ and $b$ lie $\delta$-close to their level-$n$ approximations
\begin{displaymath} [\theta]_{n} = \sum_{i = 0}^{n} \rho^{i} \cdot r_{i}\rho^{\tau} \quad \text{and} \quad [b]_{n} = \sum_{j = 0}^{n} \rho^{j} \cdot l_{j}\rho^{(1 - \tau)/2}. \end{displaymath} 
Since similarly $|a - [a]_{n}| \leq \delta$ for $a \in A$, it suffices to show that $[a]_{n} + [\theta]_{n}[b]_{n} \in [H]_{C\delta}$.

To get started with this, note that
\begin{displaymath} [\theta]_{n}[b]_{n} = \sum_{i,j = 0}^{n} \rho^{i + j} \cdot r_{i}l_{j} \rho^{(1 + \tau)/2} = \sum_{k = 0}^{2n} \rho^{k} \sum_{i + j = k} r_{i}l_{j}\rho^{(1 + \tau)/2}.  \end{displaymath} 
We claim that every sum of the form above is $\delta$-close to
\begin{equation}\label{form2} [\theta b]_{n} = \sum_{k = 0}^{n} \rho^{k} \sum_{i + j = k} r_{i}l_{j}\rho^{(1 + \tau)/2}. \end{equation} 
Write
\begin{displaymath} \theta b - [\theta b]_{n} = \sum_{k = n + 1}^{2n} \rho^{k} \sum_{i + j = k} r_{i}l_{j}\rho^{(1 + \tau)/2} =: R_{n}(b,\theta). \end{displaymath}
To see that $R_{n}(b,\theta) \leq \delta$, first observe that
\begin{displaymath} \left| \sum_{i + j = k} r_{i}k_{j}\rho^{(1+\tau)/2} \right| \leq k \cdot \max |r_{i}k_{j}\rho^{(1 + \tau)/2}| \leq 2n, \qquad k \in \{0,\ldots,2n\}, \end{displaymath} 
since $r_{i}l_{j}\rho^{(1 + \tau)/2} \leq 1$. Combining this with \eqref{form9}, and recalling \eqref{eqn:smallness}, \eqref{eqn:property of n}, we get
\begin{displaymath} R_{n}(b,\theta) \leq \delta \cdot \frac{4 n}{(\log(1/\delta))^{1/\epsilon}} \leq \delta \cdot \frac{4}{\log \log (1/\delta)} \leq \delta.  \end{displaymath}

Finally, note that every sum of the form appearing in \eqref{form2} is contained in the set
\begin{displaymath} H' := \left\{\sum_{k = 0}^{n} \rho^{k} \cdot s_{k}\rho^{(1 + \tau)/2} : s_{k} \in \N, \, 0 \leq s_{k} \leq n \cdot \rho^{-(1 + \tau)/2} \right\}. \end{displaymath} 
This is because $r_{i}l_{j} \in \{0,\ldots,\rho^{-(1 + \tau)/2}\}$ for every $(i,j)$ fixed, and the inner sum in \eqref{form2} contains $\leq k \leq n$ terms. We have now shown that $\Theta B \subset [H']_{C\delta}$. On the other hand, also $A \subset [H']_{\delta}$, because clearly $[a]_{n} \in H'$ for every $a \in A$ (recall \eqref{defA}). We have now shown that $A + \Theta B \subset [H']_{C\delta} + [H']_{C\delta} \subset [H]_{2C\delta}$, and the proof is complete. \end{proof} 

\section{Proof of Theorem \ref{main2}}\label{s4}

Fix $0 < \tau < t \leq 1$ with $t,\tau \in \Q$, and let $\delta > 0$ be so small that the conclusions of Proposition \ref{mainProp} hold. Let $A_{0},B_{0}$ be the self-similar sets from the proof of Proposition \ref{mainProp}, thus $A_{0} \subset [0,1]$ is the self-similar set generated by the similitudes
\begin{displaymath} \varphi_{j}(x) := \rho x + j\rho^{(1 + \tau)/2}, \qquad 0 \leq j \leq \rho^{-(1 + \tau)/2} - 1, \end{displaymath} 
and $B_{0} \subset [0,1]$ is the self-similar set generated by the similitudes
\begin{displaymath} \bar{\varphi}_{j}(x) := \rho x + j \rho^{(1 - \tau)/2}, \qquad 0 \leq j \leq \rho^{-(1 - \tau)/2} - 1. \end{displaymath}
Here, $\rho = (\log 1/\delta')^{-\epsilon}$, where $\epsilon := \frac{t-\tau}{6}$ and $\delta' \in [\delta^{1+\epsilon/(1+t-\epsilon)}, \delta]$ is the parameter in our deduction of Proposition \ref{mainProp} from Lemma \ref{mainLemma}. Since $(1 + \frac{\epsilon}{1+t-\epsilon})^{1/\epsilon} \le (1+\epsilon)^{1/\epsilon} \lesssim 1$, we obtain that $\rho \sim (\log 1/\delta)^{-6/(t - \tau)}$.

Note that $A_{0} \times B_{0} \subset [0,1]^{2}$ is also a self-similar set generated by $\rho^{-1}$ similitudes with contraction ratio $\rho$. Thus, for each $j \geq 1$, the set $A_{0} \times B_{0}$ is covered by a family $\mathcal{Q}_{j}$ of squares of side-length $\rho^{j}$, with cardinality $|\mathcal{Q}_{j}| = \rho^{-j}$, whose horizontal separation is $\rho^{j - 1} \cdot \rho^{(1 + \tau)/2}$ and whose vertical separation is $\rho^{j - 1} \cdot \rho^{(1 - \tau)/2}$. 

We start by claiming that the set
\begin{displaymath} K := A_{0} \times \rho^{(1 + \tau)/2}B_{0} =: A_{0} \times B \end{displaymath} 
is Ahlfors $(1,C)$-regular for an absolute constant $C > 0$. For each $j \geq 1$, note that $K$ is covered by the images of the squares $\mathcal{Q}_{j}$ under the map 
\begin{displaymath} T_{\rho}(x,y) := (x,\rho^{(1 + \tau)/2}y). \end{displaymath}
These images, denoted $\mathcal{R}_{j}$, are rectangles of dimensions $\rho^{j} \times \rho^{j} \cdot \rho^{(1 + \tau)/2}$, whose horizontal separation remains $\rho^{j - 1} \cdot \rho^{(1 + \tau)/2}$, but whose vertical separation is only $\rho^{j}$ (by horizontal and vertical separation we literally mean the horizontal and vertical separation of the centres of the rectangles). 

The set $K$ supports the measure $\mu := T_{\rho}\mu_{0}$, where $\mu_{0}$ is the canonical self-similar measure on $A \times B$. Note that $\mu(R) = \rho^{j}$ for each $R \in \mathcal{R}_{j}$. 

We now claim that $\mu(B(x,r)) \sim r$ for all $x \in K$ and $0 < r \leq 10$, where the implicit constants are absolute. Fix $0 < r \leq 10$, and let $j \geq 1$ such that $10\rho^{j} \leq r \leq 10\rho^{j - 1}$. 
\subsubsection*{Case where $10\rho^{j} \leq r \leq 10\rho^{j - 1} \cdot \rho^{(1 + \tau)/2}$} Any maximal family of rectangles in $\mathcal{R}_{j}$ such that the centres have a common $x$-coordinate is called a \emph{$j$-stack}. The horizontal separation of two distinct $j$-stacks is at least $\rho^{j - 1} \cdot \rho^{(1 + \tau)/2}$. Since $r \leq 10\rho^{j - 1} \cdot \rho^{(1 + \tau)/2}$, the disc $B(x,r)$ can meet $\leq 30$ $j$-stacks, see Figure \ref{fig1}.
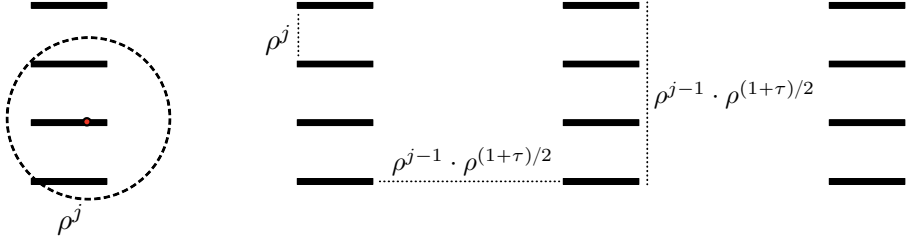
\begin{figure}[h!]
\begin{center}
\begin{overpic}[scale = 0.7]{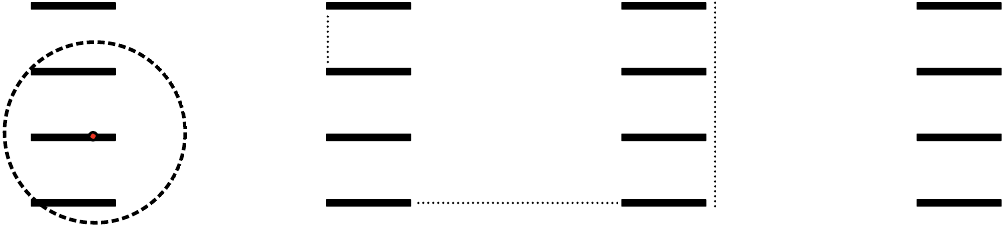}
\put(6,-3){$\rho^{j}$}
\put(29,17){$\rho^{j}$}
\put(43,3.5){\small{$\rho^{j - 1} \cdot \rho^{(1 + \tau)/2}$}}
\put(72,11){\small{$\rho^{j - 1} \cdot \rho^{(1 + \tau)/2}$}}
\end{overpic}
\caption{The case $10\rho^{j} \leq r \leq \rho^{j - 1} \cdot \rho^{(1 + \tau)/2}$.}\label{fig1}
\end{center}
\end{figure}
Since the vertical separation of $\mathcal{R}_{j}$-rectangles inside a fixed $j$-stack is $\rho^{j}$, the disc $B(x,r)$ meets $\lesssim r/\rho^{j}$ rectangles in a fixed $j$-stack. This gives the upper bound $\mu(B(x,r)) \lesssim \rho^{j} \cdot (r/\rho^{j}) = r$. But since $10\rho^{j} \leq r \leq \rho^{j - 1} \cdot 10\rho^{(1 + \tau)/2}$, also the converse holds: $B(x,r)$ contains $\gtrsim r/\rho^{j}$ rectangles in the $j$-stack containing $x$. Therefore $\mu(B(x,r)) \sim r$.
\begin{figure}[h!]
\begin{center}
\begin{overpic}[scale = 0.6]{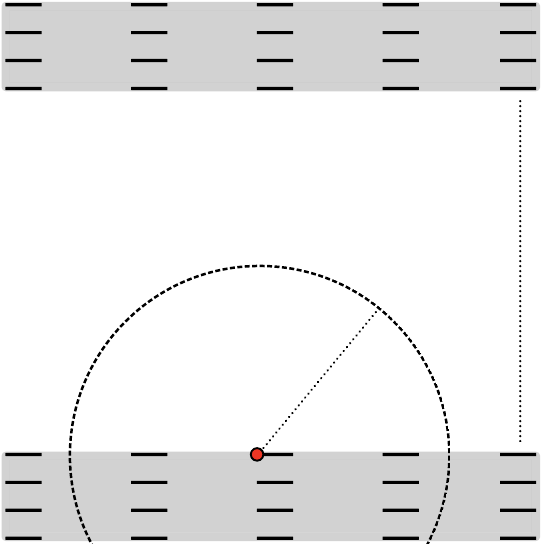}
\put(97,50){$\rho^{j - 1}$}
\put(45,19){\small{$x$}}
\put(50,26){$r$}
\put(15,7){$\mathbf{R}$}
\put(15,90){$\mathbf{R}'$}
\put(48,76){$\rho^{j - 1}$}
\end{overpic}
\caption{Case $\rho^{j} \cdot \rho^{(1 + \tau)/2} \leq r \leq 10\rho^{j - 1}$.}\label{fig2}
\end{center}
\end{figure}
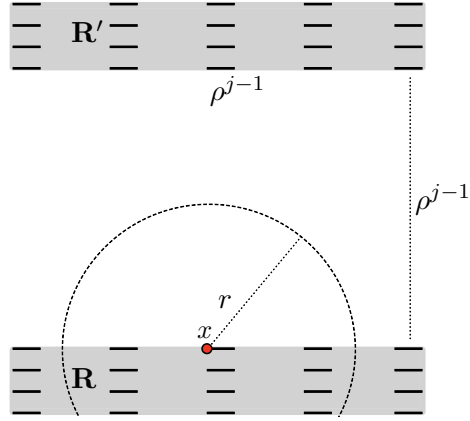
\subsubsection*{Case where $10\rho^{j - 1} \cdot \rho^{(1 + \tau)/2} \leq r \leq 10\rho^{j - 1}$} In this range the disc $B(x,r)$ will intersect multiple $j$-stacks $\mathcal{S}_{1},\ldots,\mathcal{S}_{n} \subset \mathcal{R}_{j}$: 
\begin{equation}\label{form12} n \sim \frac{r}{\rho^{j - 1} \cdot \rho^{(1 + \tau)/2}}. \end{equation} 
The key observation this time is that the vertical separation of (vertically) neighbouring rectangles $\mathbf{R},\mathbf{R}' \in \mathcal{R}_{j - 1}$ is $\rho^{j - 1}$, see Figure \ref{fig2}. Since $r \leq 10\rho^{j - 1}$, the disc $B(x,r)$ may only intersect those $\mathcal{R}_{j}$-rectangles in a fixed stack $\mathcal{S}_{i}$ which are additionally contained in $\mathbf{R} \ni x$, or at most $30$ neighbours of $\mathbf{R}$ in $\mathcal{R}_{j - 1}$. Let 
\begin{displaymath} S_{i}(\mathbf{R}') := \cup \{R \in \mathcal{R}_{j} : R \in \mathcal{S}_{i} \text{ and } R \subset \mathbf{R}'\}, \qquad \mathbf{R}' \in \mathcal{R}_{j - 1}. \end{displaymath}
Then $\mu(S_{i}(\mathbf{R}')) = \rho^{j - 1} \cdot \rho^{(1 + \tau)/2}$. This gives $\mu(B(x,r)) \lesssim r$ when combined with \eqref{form12}. The lower bound $\mu(B(x,r)) \gtrsim r$ follows by observing that $B(x,r)$ also contains $\sim r/(\rho^{j - 1} \cdot \rho^{(1 + \tau)/2})$ sets of the form $S_{i}(\mathbf{R})$. This proves the Ahlfors $(1,C)$-regularity of $\mu$.

To complete the proof of Theorem \ref{main2}, let $\nu_{0}$ be the (self-similar) $(\tau,C)$-Frostman measure from Proposition \ref{mainProp}, with $\Theta_{0} := \spt \nu_{0}$, and let $\bar{\nu}$ be the push-forward of $\nu_{0}$ under the map $\theta \mapsto \rho^{-(1 + \tau)/2}\theta$. Then $\bar{\nu}$ is a probability measure on $\R$, and 
\begin{equation}\label{form13} \bar{\nu}(B(x,r)) = \nu_{0}(B(\rho^{(1 + \tau)/2}x,\rho^{(1 + \tau)/2}r)) \lesssim \rho^{\tau(1 + \tau)/2}r^{\tau}, \qquad  x \in \R, \, r > 0. \end{equation}
The measure $\bar{\nu}$ is not quite admissible for Theorem \ref{main2}, because it is not a probability measure on $[0,1]$. To fix this, we define $\nu := \bar{\nu}([0,1])^{-1}\bar{\nu}|_{[0,1]}$. Note that
\begin{displaymath} \bar{\nu}([0,1]) = \nu_{0}([0,\rho^{(1 + \tau)/2}]) \geq \nu_{0}([0,\rho]) = \rho^{\tau}, \end{displaymath}
using the self-similar definition of $\nu_{0}$ (see \eqref{form14}). Combining this lower bound with \eqref{form13} leads to $\nu(B(x,r)) \lesssim \rho^{\tau(\tau - 1)/2}r^{\tau} \sim (\log 1/\delta)^{3\tau(1-\tau)/(t - \tau)}$ for all $x \in \R$ and $r > 0$.

Finally, writing $\spt \nu =: \Theta \subset \rho^{-(1 + \tau)/2}\Theta_{0}$, we have $A_{0} + \Theta B \subset A_{0} + \Theta_{0}B_{0}$ (recalling $B = \rho^{(1 + \tau)/2}B_{0}$), and therefore 
\begin{displaymath} N_\delta (\pi_{\theta}(K)) \leq N_\delta (A_{0} + \Theta B) \leq N_\delta (A_{0} + \Theta_{0}B_{0}) \leq \delta^{(1 - t)/2}, \qquad \theta \in \Theta \end{displaymath}
by Proposition \ref{mainProp}. This completes the proof of Theorem \ref{main2}.

\bibliographystyle{plain}
\bibliography{references}

\def\cprime{$'$}
\begin{thebibliography}{1}

\bibitem{MR4745881}
Damian D{\k a}browski, Tuomas Orponen, and Hong Wang.
\newblock How much can heavy lines cover?
\newblock {\em J. Lond. Math. Soc. (2)}, 109(5):Paper No. e12910, 33, 2024.

\bibitem{2024arXiv240411179F}
Jonathan~M. {Fraser} and Ana~E. {de Orellana}.
\newblock {A Fourier analytic approach to exceptional set estimates for
  orthogonal projections}.
\newblock {\em arXiv e-prints}, page arXiv:2404.11179, April 2024.

\bibitem{Ho}
Michael Hochman.
\newblock On self-similar sets with overlaps and inverse theorems for entropy.
\newblock {\em Ann. of Math. (2)}, 180(2):773--822, 2014.

\bibitem{Ka}
Robert Kaufman.
\newblock On {H}ausdorff dimension of projections.
\newblock {\em Mathematika}, 15:153--155, 1968.

\bibitem{MR4388762}
Tuomas Orponen.
\newblock On arithmetic sums of {A}hlfors-regular sets.
\newblock {\em Geom. Funct. Anal.}, 32(1):81--134, 2022.

\bibitem{2024arXiv241006872O}
Tuomas {Orponen}.
\newblock {On the projections of Ahlfors regular sets in the plane}.
\newblock {\em arXiv e-prints}, page arXiv:2410.06872, October 2024.

\bibitem{2023arXiv230110199O}
Tuomas {Orponen} and Pablo {Shmerkin}.
\newblock {Projections, Furstenberg sets, and the $ABC$ sum-product problem}.
\newblock {\em arXiv e-prints}, page arXiv:2301.10199, January 2023.

\bibitem{2023arXiv230808819R}
Kevin {Ren} and Hong {Wang}.
\newblock {Furstenberg sets estimate in the plane}.
\newblock {\em arXiv e-prints}, page arXiv:2308.08819, August 2023.

\bibitem{Wolff99}
Thomas Wolff.
\newblock Recent work connected with the {K}akeya problem.
\newblock In {\em Prospects in mathematics ({P}rinceton, {NJ}, 1996)}, pages
  129--162. Amer. Math. Soc., Providence, RI, 1999.

\end{thebibliography}

\end{document}